\UseRawInputEncoding
\documentclass[leqno]{amsart}

\usepackage{stmaryrd,graphicx}
\usepackage{amssymb,mathrsfs,amsmath,amscd,amsthm,color}
\usepackage{float}
\usepackage{mathabx}
\usepackage[all,cmtip]{xy}
\DeclareMathAlphabet{\mathpzc}{OT1}{pzc}{m}{it}
\usepackage{amsfonts,latexsym,wasysym}

\theoremstyle{theorem}
\newtheorem{theorem}{Theorem}
 
\theoremstyle{definition}
\newtheorem*{definition}{Definition}
\newtheorem{remark}{Remark}
\newtheorem{example}[theorem]{Example}
\newtheorem{lemma}[theorem]{Lemma}
\newtheorem{proposition}[theorem]{Proposition}

\newcommand{\ui}{[0,1]}
\newcommand{\ov}{\overline}
\newcommand{\scrl}{\mathscr{L}}
\newcommand{\scru}{\mathscr{U}}
\newcommand{\scrv}{\mathscr{V}}
\newcommand{\bbn}{\mathbb{N}}
\newcommand{\bbq}{\mathbb{Q}}
\newcommand{\bbr}{\mathbb{R}}

\begin{document}

\title{Constructing arcs from paths using Zorn's Lemma}
\author[J. Brazas]{Jeremy Brazas}
\address{West Chester University\\ Department of Mathematics\\
West Chester, PA 19383, USA}
\email{jbrazas@wcupa.edu}


\date{\today}

\begin{abstract}
It is a well-known fact that every path-connected Hausdorff space is arcwise connected. Typically, this result is viewed as a consequence of a sequence of fairly technical results from continuum theory. In this note, we exhibit a direct and simple proof of this statement, which makes explicit use of Zorn's Lemma. Additionally, by carefully breaking down the proof, we identify a modest improvement to a class of spaces relevant to algebraic topology.
\end{abstract}

\maketitle

The following theorem is a well-established and commonly used result from general topology, which can be found in many introductory and reference textbooks on the subject, e.g. \cite{Eng89,HS,Moore,Nadler,willard}.

\begin{theorem}\label{hausdorfftheorem}
Every path-connected Hausdorff space is arcwise connected.
\end{theorem}

This result is typically proved using a combination of basic general topology, metrization theory, and the ``Arcwise Connectedness Theorem," which asserts that any locally connected continuum (connected, compact metric space), is arcwise connected. To prove the Arcwise Connectedness Theorem, one is required one to construct an arc between two points without prior knowledge that even one non-constant path in the space exists. Consequently, proofs of this theorem are quite delicate. In fact, several authors of well-known topology textbooks have made critical oversights in the proof \cite{Ball}. It is possible that the popularity of treating Theorem \ref{hausdorfftheorem} as a consequence of the Arcwise Connectedness Theorem, has shadowed elementary proofs and potential generalizations of Theorem \ref{hausdorfftheorem}. The author believes that having a variety of proofs for one result can often be beneficial. In this expository note, we give a direct, elementary proof of Theorem \ref{hausdorfftheorem} using Zorn's Lemma. We find this proof to be conceptually simpler than other known proofs and a nice example of how the axiom of choice can sometimes ease technicalities one might face if intentionally trying to avoid the axiom of choice. 

The primary difficulty in proving Theorem \ref{hausdorfftheorem} directly is ensuring that one can always replace a path with an injective path. We must begin with an arbitrary non-constant path $\alpha:\ui\to X$ (which may be space-filling) and search for a collection of pairwise-disjoint intervals $(a,b)\subseteq \ui$ such that $\alpha|_{[a,b]}$ is a loop. We refer to such a collection as a \textit{loop-cancellation of }$\alpha$. Provided that a loop-cancellation is maximal in the partial order of all loop-cancellations of $\alpha$, we may then ``pinch off" or ``delete" the corresponding subloops to obtain an injective path. To verify the existence of a maximal loop-cancellation, we employ the axiom of choice in the form of Zorn's Lemma. 

The author knows of two other published proofs of Theorem \ref{hausdorfftheorem}, which also take the approach of deleting loops from paths. We take a moment to mention a few things about them.

First, is S.B. Nadler's proof in \cite[Theorem 8.23]{Nadler}. Here, Theorem \ref{hausdorfftheorem} is proved for locally path-connected continua and then extended to all Hausdorff spaces using a metrization theorem. However, much like other textbook proofs, Nadler's line of argument is indirect, relying on a variety of far more general results, such as the``Maximum-Minimum Theorem" \cite[Exercise 4.34]{Nadler}. The Maximum-Minimum Theorem uses the compactness of the hyperspace of compact subsets of $\ui$ to allow one to find, what we are calling, a ``maximal loop-cancellation" without appealing to the axiom of choice.

Second, is a beautiful and direct proof by R. B\"{o}rger in the mostly overlooked note \cite{Borger}. This proof, as far as the author can tell, is the only published direct proof of Theorem \ref{hausdorfftheorem}. Rather than focusing on open sets as we do, B\"{o}rger takes the dual approach of constructing a nested sequence of closed sets $A_1\supseteq A_2\supseteq \cdots $ and, essentially, works to show the components of the complement of $\bigcap_{n}A_n$ form a maximal loop-cancellation. B\"{o}rger mentions in the introduction ``I am indebted to K.P. Hart for some simplifications, particularly for avoiding use of Zorn's lemma." Hence, we do not doubt that B\"{o}rger knew of a proof, which is similar in spirit to the one in this note.

Certainly, one could argue that our proof is logically redundant, because it unnecessarily uses the axiom of choice. However, there are often many benefits to exploring different proofs of well-known results. Moreover, there is a certain conceptual simplicity to our proof that the author does not find in any of the other methods of proof. Indeed, some of the technical ``weight" of other proofs appears to be absorbed by Zorn's Lemma. Those who freely use the axiom of choice may find it preferable.

Finally, in Section \ref{sectionfinal}, we note that by closely breaking down our ``from scratch" proof, it is possible to prove a modest generalization of Theorem \ref{hausdorfftheorem} that replaces the ``Hausdorff" hypothesis with a strictly weaker property that is relevant to categories commonly used in algebraic topology (see Theorem \ref{mainthm}).

\section{A proof of Theorem \ref{hausdorfftheorem}}\label{sectionproof}

First, we establish some notation and terminology. A \textit{path} is a continuous function $\alpha:[a,b]\to X$ from a closed real interval $[a,b]$. If $\alpha(a)=x$ and $\alpha(b)=y$, we say that $\alpha$ is a \textit{path from $x$ to} $y$. If $\alpha(a)=\alpha(b)$, we will refer to $\alpha$ as a \textit{loop}. Often, we will use the closed unit interval $\ui$ to be the domain of a path. Given paths $\alpha:[a,b]\to X$ and $\beta:[c,d]\to X$, we write $\alpha\equiv\beta$ if $\alpha=\beta\circ \phi$ for some increasing homeomorphism $\phi: [a,b]\to [c,d]$ and we say that $\alpha$ \textit{is a reparameterization of} $\beta$. 
If $\alpha:\ui\to X$ is a path, then $\alpha^{-}(t)=\alpha(1-t)$ is the reverse path. If $\alpha,\beta:\ui\to X$ are paths with $\alpha(1)=\beta(0)$, then $\alpha\cdot\beta:\ui\to X$ denotes the usual concatenation of the two paths.

\begin{definition}
A space $X$ is
\begin{enumerate}
\item \textit{path connected} if whenever $x,y\in X$, there exists a path from $x$ to $y$,
\item \textit{arcwise connected} if whenever $x\neq y$ in $X$, there exists a path $\alpha:\ui\to X$ from $x$ to $y$, which is a topological embedding, i.e. a homeomorphism onto its image.
\end{enumerate}
\end{definition}

Certainly, arcwise connected $\Rightarrow$ path connected. We do not consider ``monotone" paths (those which have connected fibers) as being of separate interest from injective-paths since a quotient map argument shows that every monotone path in a $T_1$ space may be replaced by an injective path with the same image.

Toward our proposed generalization of Theorem \ref{hausdorfftheorem}, we give the following definition.

\begin{definition}
We say that a space $X$ \textit{permits loop deletion} if whenever $\alpha:\ui\to X$ is a path and there exist $0\leq \cdots \leq a_3\leq a_2\leq a_1<b_1\leq b_2\leq b_3\leq \cdots \leq 1$ such that $\{a_n\}\to 0$, $\{b_n\}\to 1$, and $\alpha(a_n)=\alpha(b_n)$ for all $n\in\bbn$, then $\alpha(0)=\alpha(1)$.
\end{definition}

Intuitively, if $X$ permits loop deletion, then the scenario illustrated in Figure \ref{tfig} cannot occur, that is, there cannot exist paths $\alpha,\beta:\ui\to X$ such that $\alpha(1/n)=\beta(1/n)$ for all $n\in\bbn$ and $\alpha(0)\neq \beta(0)$.

\begin{figure}[h]
\centering \includegraphics[height=1.8in]{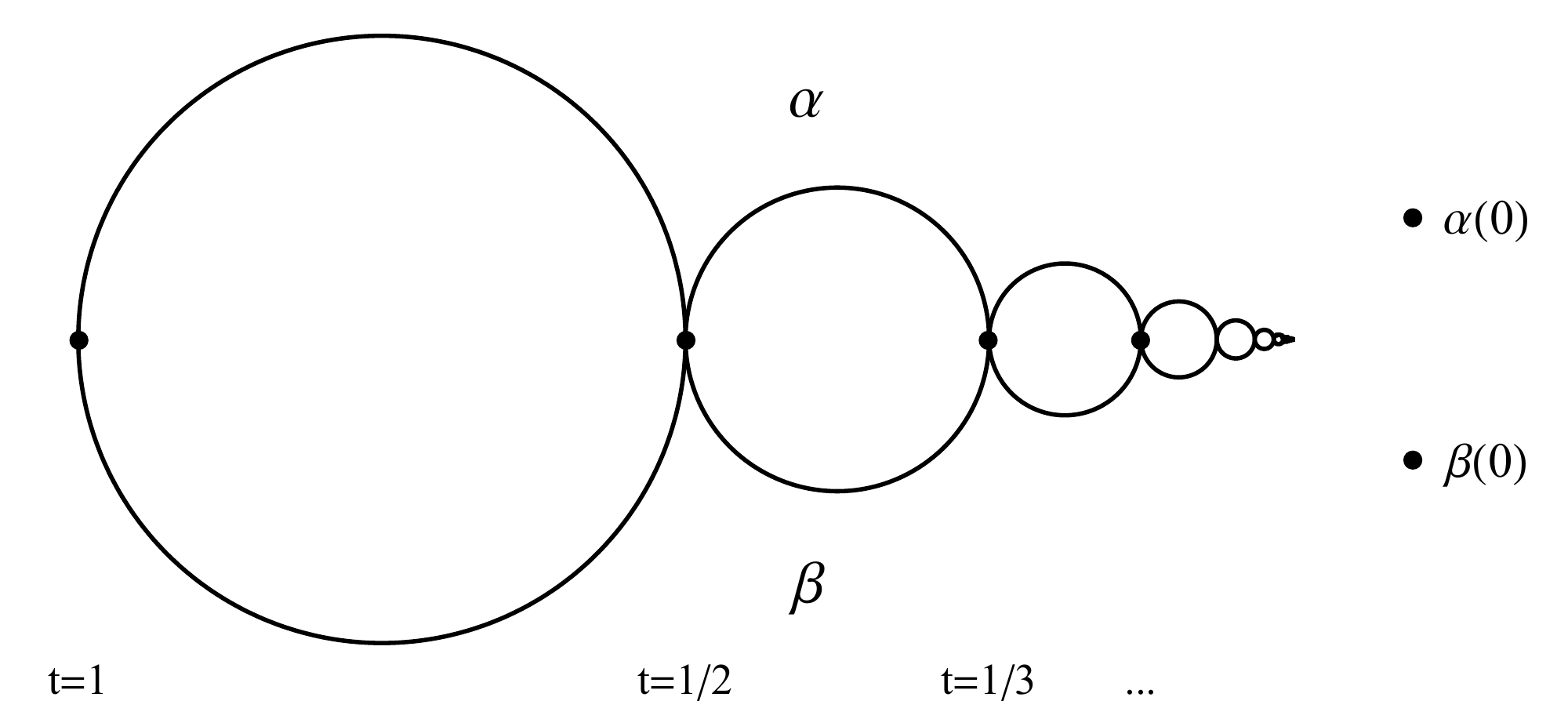}
\caption{\label{tfig}In a space that permits loop deletion, there cannot be two paths $\alpha,\beta$ which agree on a sequence converging to $0$ but for which $\alpha(0)\neq\beta(0)$.}
\end{figure}

\begin{remark}
If a space $X$ is Hausdorff, then convergent sequences in $X$ have unique limits. Hence, every Hausdorff space permits loop deletion. 
\end{remark}

Given a path $\alpha:\ui\to X$, a \textit{loop-cancellation of }$\alpha$ is a set $\scru$ of pairwise-disjoint open intervals in $\ui$ such that for each $(a,b)\in \scru$, we have $\alpha(a)=\alpha(b)$, i.e. such that $\alpha|_{[a,b]}$ is a loop (see Figure \ref{fig2}). We endow each loop-cancellation with the linear ordering inherited from the ordering of $\ui$. Let $\scrl(\alpha)$ denote the set of all loop-cancellations of $\alpha$. We give $\scrl(\alpha)$ the following partial order: $\scrv\geq \scru$ in $\scrl(\alpha)$ if for each $U\in \scru$, there exists $V\in\scrv$ such that $U\subseteq V$. Clearly, the empty set is minimal in $\scrl(\alpha)$. We say a loop-cancellation is \textit{maximal in $\scrl(\alpha)$} if it is maximal with respect to this partial order on $\scrl(\alpha)$. If $\alpha$ is itself a loop, then $\{(0,1)\}$ is a maximal loop-cancellation of $\alpha$.

To construct injective paths, we wish to ``delete" loops occurring as subpaths of $\alpha$. Formally, this must be done within the domain by collapsing the closure of each element of a loop-cancellation to a single point so that the resulting quotient space is homeomorphic to $\ui$. The next definition identifies when such an operation is possible.

\begin{definition}
We say that a loop-cancellation $\scru\in\scrl(\alpha)$ is \textit{collapsible} if $\scru\neq \{(0,1)\}$ and if the elements of $\scru$ have pairwise disjoint closures.
\end{definition}

\begin{remark}\label{collapsableremark}
If $\alpha$ is not a loop and $\scru$ is maximal in $\scrl(\alpha)$, then $\scru$ is necessarily collapsible. Otherwise, we would have $(a,b),(b,c)\in\scru$, which implies $\alpha(a)=\alpha(b)=\alpha(c)$. We could then replace these two with $(a,c)$ to form a loop-cancellation that is greater in $\scrl(\alpha)$.
\end{remark}

Let $\scru$ be a collection of open intervals in $(0,1)$ with pairwise-disjoint closures. Basic constructions from real analysis give the existence of non-decreasing, continuous surjections $\Gamma_{\scru}:\ui\to \ui$ that are constant on the closure of each set $U\in\scru$ and which are strictly increasing on $\ui\backslash \bigcup\{\ov{I}\mid I\in\scru\}$. We refer to such a function $\Gamma_{\scru}$ as a \textit{collapsing function for }$\scru$. For example, if $\scru$ is the set of components of the complement of the ternary Cantor set, then the ternary Cantor function is a collapsing function for $\scru$. Note that $\Gamma_{\scru}$ is not unique to $\scru$ but if $\Gamma_{\scru}$ and $\Gamma_{\scru}'$ are two collapsing functions for $\scru$, then we have $h\circ \Gamma_{\scru}=\Gamma_{\scru}'$ for some increasing homeomorphism $h:\ui\to \ui$.

\begin{definition}
Suppose $\scru\in\scrl(\alpha)$ is collapsible and let $\Gamma_{\scru}$ be a collapsing function for $\scru$. We write $\alpha_{\scru}:\ui\to X$ for the path, which agrees with $\alpha$ on $\ui\backslash \bigcup\scru$ and such that $\alpha|_{\scru}$ is constant on each $I\in\scru$, that is, $\alpha_{\scru}(\overline{I})=\alpha(\partial I)$ for each $I\in\scru$. By the universal property of the quotient map $\Gamma_{\scru}$, there exists a unique path $\beta:\ui\to X$ satisfying $\beta\circ \Gamma_{\scru}=\alpha_{\scru}$, which we refer to as a \textit{$\scru$-reduction of $\alpha$.}
\end{definition}

Since any two collapsing functions for $\scru$ differ by a homeomorphism, it follows that if $\beta$ and $\beta '$ are two $\scru$-reductions of $\alpha$ (constructed using different collapsing functions), then $\beta\equiv \beta '$.

\begin{figure}[h]
\centering \includegraphics[height=1.7in]{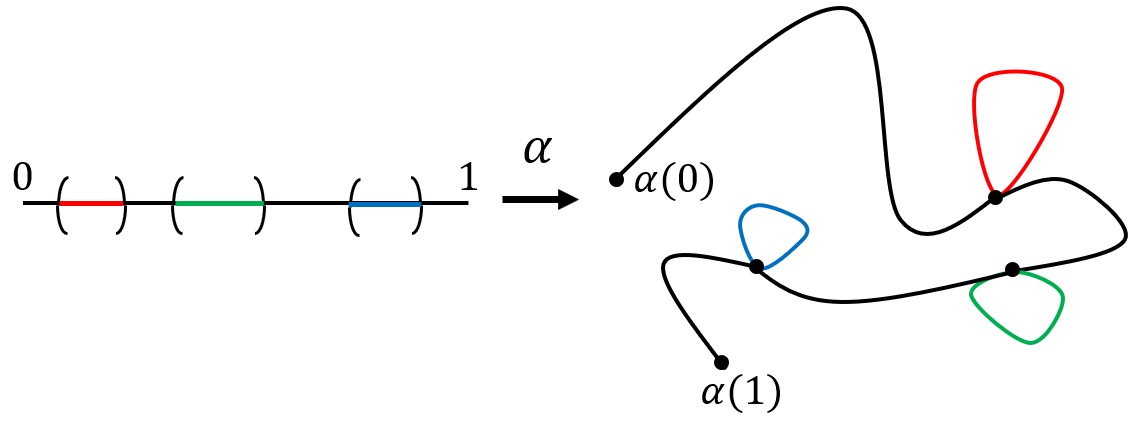}
\caption{\label{fig2}A loop-cancellation $\scru$ of a path $\alpha$ consisting of three intervals each of which is mapped to a loop. In this case, the loop-cancellation is maximal and a parameterization of the black arc is a $\scru$-reduction of $\alpha$.}
\end{figure}

\begin{proposition}\label{loopreductionisinjectiveprop}
If $\scru\in\scrl(\alpha)$ is maximal and $\beta$ is a $\scru$-reduction of $\alpha$, then $\beta$ is injective.
\end{proposition}

\begin{proof}
We prove the contrapositive. Suppose $0\leq a<b\leq 1$, with $\beta(a)=\beta(b)$. Fix collapsing function $\Gamma_{\scru}$ for $\scru$ such that $\beta\circ\Gamma_{\scru}=\alpha_{\scru}$. Since $\Gamma_{\scru}$ is non-decreasing, $\Gamma_{\scru}^{-1}([a,b])$ is a closed interval, call it $[c,d]$. Notice that for each $I\in\scru$, either $I\subseteq (c,d)$ or $I\cap [c,d]=\emptyset$. Now $c,d\notin\bigcup\scru$ and $\alpha_{\scru}(c)=\beta(a)=\beta(b)=\alpha_{\scru}(d)$. Since $\alpha$ agrees with $\alpha_{\scru}$ on $\ui\backslash\bigcup\scru$, we have $\alpha(c)=\alpha(d)$. Now $\scrv=\{I\in\scru\mid I\cap [c,d]=\emptyset\}\cup\{(c,d)\}$ is a loop-cancellation for $\alpha$ that is greater than $\scru$ in $\scrl(\alpha)$. Hence, $\scru$ is not maximal.
\end{proof}

It is not entirely obvious that a maximal loop-cancellation must exist for an arbitrary path. Indeed, it is possible for distinct loop-cancellations to nest and overlap in complicated ways. Compounding the issue is the fact that different maximal loop-cancellations may result in different reductions. For example, suppose $\beta,\gamma:\ui\to X$ are injective paths for which $\beta((0,1))\cap \gamma((0,1))=\emptyset$, $\beta(0)=\gamma(0)$, and $\beta(1)=\gamma(1)$. If $\alpha=(\beta\cdot \beta^{-})\cdot \gamma$, then $\scru=\{(0,1/2)\}$ is maximal in $\scrl(\alpha)$ with $\gamma$ as a $\scru$-reduction and $\scrv=\{(1/4,1)\}$ is maximal with $\beta$ as a $\scrv$-reduction. Based on these observations, it is natural to attempt an application of Zorn's Lemma.

\begin{lemma}\label{maximalloopcancellationlemma}
If $X$ permits loop deletion then for every path $\alpha:\ui\to X$, there exists a maximal loop-cancellation $\scrv\in\scrl(\alpha)$.
\end{lemma}

\begin{proof}
The conclusion is clear if $\alpha$ is a loop so we assume $\alpha(0)\neq \alpha(1)$. The lemma will follow from Zorn's Lemma once we show that every linearly ordered suborder of $\scrl(\alpha)$ has an upper bound. Suppose $S\subseteq\scrl(\alpha)$ is linearly ordered when given the order inherited from $\scrl(\alpha)$. Now $V=\bigcup_{\scru\in S}(\bigcup \scru)$ is an open subset of $(0,1)$. Let $\scrv$ denote the set of connected components of $V$. To show that $S$ has an upper bound in $\scrl(\alpha)$, it suffices to show that $\scrv\in \scrl(\alpha)$. Let $(c,d)\in\scrv$. If $(c,d)\in\scru$ for some $\scru\in S$, it is clear that $\alpha(c)=\alpha(d)$. Suppose that $(c,d)\notin \scru$ for any $\scru\in S$. 

Pick $c<\cdots <a_3<a_2<a_1<b_1<b_2<b_3<\cdots <d$ where $\{a_{n}\}\to c$ and $\{b_n\}\to d$. Fixing $n\in\bbn$, $\bigcup S$ is an open cover of the closed interval $[a_n,b_n]$ and so we may find finitely many $I_{n,1},I_{n,2},\dots, I_{n,k_n}\in \bigcup S$, which cover $[a_n,b_n]$. Find $\mathscr{W}_{n,j}\in S$ with $I_{n,j}\in \mathscr{W}_{n,j}$. Since $S$ is linearly ordered, we may define $\scru_n$ to be the maximum of $\{\mathscr{W}_{n,1},\mathscr{W}_{n,2},\dots , \mathscr{W}_{n,k_n}\}$ in $S$. Since $I_{n,j}\subseteq\bigcup\mathscr{W}_{n,j}\subseteq \bigcup\scru_n$ for all $j$, it follows that $[a_n,b_n]\subseteq \bigcup\scru_n$. For some interval $(c_n,d_n)\in \scru_n$ we have $[a_n,b_n]\subseteq (c_n,d_n)$. Moreover, since $(c_n,d_n)$ meets the connected connected component $(c,d)$ of $V$, we have $(c_n,d_n)\subseteq (c,d)$.

We now have $c\leq \cdots\leq  c_3\leq c_2\leq c_1<d_1\leq d_2\leq d_3\leq \cdots \leq d$ where $\{c_n\}\to c$ and $\{d_n\}\to d$. Moreover, since $(c_n,d_n)\in \scru_n\in S$, we have $\alpha(c_n)=\alpha(d_n)$ for all $n\in\bbn$. Finally, since we are assuming that $X$ permits loop deletion, it follows that $\alpha(c)=\alpha(d)$. Therefore, $\scrv$ is a loop cancellation, i.e. $\scrv\in\scrl(\alpha)$.
\end{proof}

\begin{proof}[Proof of Theorem \ref{hausdorfftheorem}]
Suppose $X$ is a path connected Hausdorff space and fix $x,y\in X$ with $x\neq y$. Find a path $\alpha:\ui\to X$ from $x$ to $y$. According to Lemma \ref{maximalloopcancellationlemma}, there exists a maximal loop-cancellation $\scrv\in\scrl(\alpha)$. Since $\scrv$ must be collapsible (Remark \ref{collapsableremark}), we may choose a collapsing function $\Gamma_{\scrv}$ for $\scrv$. Now the $\scrv$-reduction $\beta:\ui\to X$ satisfying $\beta\circ\Gamma_{\scrv}=\alpha_{\scrv}$ is injective by Proposition \ref{loopreductionisinjectiveprop}. Since $\Gamma_{\scrv}:\ui\to\ui$ is a non-decreasing surjection, we have $\beta(0)=x$ and $\beta(1)=y$. Moreover, since $\ui$ is compact and $X$ is Hausdorff, the continuous injection $\beta$ is a homeomorphism onto its image. This proves $X$ is arcwise connected.
\end{proof}

\section{What other spaces permit loop deletion?}\label{sectionfinal}

There are, in fact, some commonly used spaces that permit loop deletion but which are not necessarily Hausdorff. Such spaces become particularly relevant when general constructions fail to be closed under the Hausdorff property.

\begin{definition}
We say a space $X$ is
\begin{enumerate}
\item \textit{weakly Hausdorff} if for every map $f:K\to X$ from a compact Hausdorff space $K$, the image $f(K)$ is closed in $X$,
\item \textit{$\Delta$-Hausdorff} if for every path $\alpha:\ui\to X$, the image $\alpha(\ui)$ is closed in $X$.
\end{enumerate}
\end{definition}

The following implications hold: Hausdorff $\Rightarrow$ weakly Hausdorff $\Rightarrow$ $\Delta$-Hausdorff $\Rightarrow$ $T_1$. The weakly Hausdorff property is particularly relevant to algebraic topology, where it provides a suitable ``separation axiom" within the category of compactly generated spaces \cite{Strickland}. As noted in \cite{McCord}, if one is performing ``gluing" constructions involving quotient topologies in algebraic topology, the weakly Hausdorff property is often preferable over the Hausdorff property since many such constructions preserve the former property but not the latter. The $\Delta$-Hausdorff property is the analogue in the category of ``$\Delta$-generated spaces" \cite{CSWdiff,FRdirected} and offers the same kind of conveniences.

\begin{example}
The one-point compactification $X^{\ast}$ of any non-locally compact Hausdorff space $X$ is weakly Hausdorff but not Hausdorff. This occurs, for example, if $X$ is $\bbq$, $\bbr^{\omega}$ with the product topology, or $\ui^{\omega}$ in the uniform topology. One can attach arcs or other spaces to $X^{\ast}$ to obtain path-connected examples. Hence, there are many $\Delta$-Hausdorff spaces that are not Hausdorff.
\end{example}

To extend Theorem \ref{hausdorfftheorem}, we check that all of the ingredients of the proof work for $\Delta$-Hausdorff spaces.

\begin{lemma}\label{deltaclosedlemma}
If $X$ is $\Delta$-Hausdorff, then $X$ permits loop deletion.
\end{lemma}

\begin{proof}
Suppose, to obtain a contradiction, that $\alpha(0)=x\neq y=\alpha(1)$. Set $x_n=\alpha(a_n)=\alpha(b_n)$ and note that $\{x_n\}$ converges to both $x$ and $y$ in $X$. If the sequences $\{a_n\}$ and $\{b_n\}$ stabilize to $0$ and $1$ respectively, then $\alpha(b_n)=x$ for sufficiently large $n$. Therefore, the constant sequence at $x$ converges to $y$. Since every $\Delta$-Hausdorff space is $T_1$, we obtain a contradiction. 

We now assume that one of the sequences $\{a_n\}$ or $\{b_n\}$ is not eventually constant. Without loss of generality, we may assume $\{a_n\}$ is not eventually constant. Thus $0<a_n$ for all $n\in\bbn$. Since $X$ is $\Delta$-Hausdorff, the sets $\alpha([0,a_n])$, $n\in\bbn$ are closed in $X$ and contain $\{x_m\mid m\geq n\}$. Since $\{x_m\}_{m\geq n}\to y$, we have $y\in \alpha([0,a_n])$. Since $\{a_n\}$ is non-increasing an converges to $0$, we may find a decreasing sequence $\{t_j\}$ in $(0,a_1]$ that converges to $0$ and for which $\alpha(t_j)=y$ for all $j\in\bbn$. However, since $\{t_j\}\to 0$, it follows that the constant sequence at $y$ converges to $x$. Thus $X$ is $T_1$ and $x\in\ov{\{y\}}$; a contradiction.
\end{proof}

Since we are no longer assuming $X$ is Hausdorff, the usual Closed Mapping Theorem (the continuous image of a compact space in a Hausdorff space is closed) does not apply. Hence, we must also make sure that an injective path in a $\Delta$-Hausdorff space is an embedding.

\begin{lemma}\label{injtoembeddinglemma}
If $X$ is $\Delta$-Hausdorff, then every injective-path in $X$ is a closed embedding.
\end{lemma}

\begin{proof}
Let $\alpha:\ui\to X$ be an injective-path and let $C\subseteq \ui$ be closed. Write $C=\bigcap_{n\in\bbn}F_n$ where $F_n$ is a finite, disjoint union of closed intervals. Since $X$ is $\Delta$-Hausdorff, if $[a,b]$ is a component of $F_n$, then $\alpha([a,b])$ is closed in $X$. Therefore, $\alpha(F_n)$ is closed in $X$ for all $n\in\bbn$. Since $\alpha$ is injective, we have $\alpha(C)=\bigcap_{n\in\bbn}\alpha(F_n)$ and we conclude that $\alpha(C)$ is closed in $X$.
\end{proof}

With Lemmas \ref{deltaclosedlemma} and \ref{injtoembeddinglemma} established, the same proof used in Section \ref{sectionproof} gives the following generalization of Theorem \ref{hausdorfftheorem}.

\begin{theorem}\label{mainthm}
Every path-connected, $\Delta$-Hausdorff topological space is arcwise connected.
\end{theorem}

Upon inspection, one can see that all parts of B\"{o}rger's proof of Theorem \ref{hausdorfftheorem} also goes through for $\Delta$-Hausdorff spaces. Hence, Theorem \ref{mainthm} can be proven without appealing to the axiom of choice.

\begin{example}
Even with the weakened hypothesis, the converse of Theorem \ref{mainthm} is certainly not true. For a counterexample, let $X$ be the quotient space $[-1,1]/\mathord{\sim}$ where $-\frac{n}{n+1}\sim \frac{n}{n+1}$ for $n\in\bbn$ (this space is precisely illustrated in Figure \ref{tfig}). As a quotient of a closed interval, $X$ is $\Delta$-generated. However, one can show that $X$ is not arcwise connected and therefore is not $\Delta$-Hausdorff. Let $a,b$ be the images of $-1,1$ in $X$ respectively and let $Y$ be the space obtained by attaching a copy of $[0,1]$ to $X$ by identifying $0\sim a$ and $1\sim b$. Now $Y$ is arcwise connected but it is not $\Delta$-Hausdorff.
\end{example}

Indeed, it is unrealistic to hope that there is some simple topological property $P$ that gives ``path connected $+$ $P$ $\Leftrightarrow$ arcwise connected" and whose definition doesn't involving quantifying over all paths in the space. However, it is possible to show that $X$ is $\Delta$-Hausdorff if and only if every non-loop path in $X$ has an injective $\scru$-reduction. Since we have already proven the ``hard" direction in this note, we'll leave the converse as an exercise.

\end{document}